\documentclass[12pt]{amsart}
\usepackage[colorlinks=true, pdfstartview=FitV, linkcolor=blue, citecolor=blue, urlcolor=blue]{hyperref}
\usepackage{amssymb}
\calclayout 

\def\R{\mathbb{R}}
\def\C{\mathbb {C}}
\def\Z{\mathbb{Z}}
\def\M{\mathcal M}
\def\Z{\mathbb Z}
\def\lie#1{\mathfrak{ #1}}
\def\lieh{\lie h}
\def\liek{\lie k}
\def\liep{\lie p}
\def\lieq{\lie q}
\def\liea{\lie a}
\def\lieg{\lie g}
\def\liem{\lie m}
\def\lien{\lie n}
\def\lieu{\lie u}
\def\lier{\lie r}
\def\liet{\lie t}
\def\lieb{\lie b}

\def\st{($\sigma$,$\theta$)}

\def\inv{^{-1}}

\def\sm#1#2#3#4{\left(\smallmatrix #1 & #2 \\ #3 & #4\endsmallmatrix\right)}

\newcommand{\Int}{\operatorname{Int}}
\newcommand{\Aut}{\operatorname{Aut}}

\newcommand{\SL}{\operatorname{SL}}
\newcommand{\Sp}{\operatorname{Sp}}
\newcommand{\SU}{\operatorname{SU}}
\newcommand{\U}{\operatorname{U}}
\newcommand{\SO}{\operatorname{SO}}
\newcommand{\Orth}{\operatorname{O}}

\newcommand{\Ad}{\operatorname{Ad}}
\newcommand{\ad}{\operatorname{ad}}

\newcommand{\conj}{\operatorname{conj}}
\renewcommand{\SS}{{\mathcal S}}
\newcommand{\Ker}{\operatorname{Ker}}

\def\quot#1#2{#1/\!\!/#2}

\def\c{_\C}

\def\GL{\operatorname{GL}}
\def\phi{{\varphi}}

\numberwithin{equation}{subsection}

\newtheorem{theorem}[subsection]{Theorem}
\newtheorem{lemma}[subsection]{Lemma}
\newtheorem{proposition}[subsection]{Proposition}
\newtheorem{corollary}[subsection]{Corollary}

\theoremstyle{definition}

\newtheorem{definition}[subsection]{Definition}

\theoremstyle{remark}

\newtheorem{remark}[subsection]{Remark}

\newtheorem{example}[subsection]{Example}

\title[Real double coset spaces and their invariants]{\boldmath Real double coset spaces and their invariants} 
\author{Aloysius G. Helminck}
\thanks{First author is partially supported by N.S.F. Grant DMS-0532140}
\address{Department of Mathematics\\
North Carolina State University\\
Raleigh, N.C., 27695}
\email{loek@math.ncsu.edu}
\author{Gerald W. Schwarz}
\thanks{Second author is partially supported by N.S.A. Grant H98230-06-1-0023}
\address{Department of Mathematics\\
Brandeis University\\
Waltham, MA 02454-9110}
\email{schwarz@brandeis.edu}

\begin{document}

\begin{abstract} Let $G$ be a real form of a complex reductive group. Suppose that we are given involutions $\sigma$ and $\theta$ of $G$. Let $H=G^\sigma$ denote the fixed group of $\sigma$  and let $K=G^\theta$ denote the fixed group  of $\theta$. We are interested in calculating the double coset space $H\backslash G / K$. We use moment map and invariant theoretic techniques to calculate the double cosets, especially the ones that are closed. One salient point of our results is a stratification of a quotient of a compact torus over which the closed double cosets fiber as a collection of trivial bundles.
\end{abstract}

\maketitle

 \section{Introduction}  \label{introduction}

Let $U$ be a compact connected Lie group and let $U\c$ denote its complexification. Assume that we have a  real form $G$  of $U\c$ with corresponding real involution $\phi$ which preserves $U$. Assume that we have holomorphic  involutions $\sigma$ and $\theta$  of $U\c$, commuting with $\phi$, such that they generate a finite group of automorphisms of the connected center of $U\c$. Let $H$  be an open subgroup of $G^\sigma$, the fixed points of $\sigma$,  and let $K$ denote an open subgroup of $G^\theta$.   We are interested in the space of double cosets $H\backslash G/K$. As in \cite[\S 2]{Matsuki97} one can reduce to the case that 
\begin{equation} \label{eqn:assumption}
G=HG^0K. 
\end{equation}

For most of this paper we will assume that $K=G^\theta$ and explain later how this assumption can be removed. Then, as in our previous work \cite{Helm-Schwarz01}, we identify $G/K$ with a submanifold $X$ of  $G$ via the mapping $g\mapsto\beta(g):=g\theta(g\inv)$. Via this identification, the $H$ action on $G/K$ becomes the $*$-action on $X$ where $h*x:=hx\theta(h)\inv$, $h\in H$, $x\in X$. We show that there is a quotient $\quot X{H}$ parameterizing the closed orbits (then one can, in principal, determine all orbits). We can assume that the Cartan involution $\delta$ of $U\c$ commutes with $\sigma$ and $\theta$ (\S \ref{sec:Cartan}). Let $G_0$ denote $G\cap U$. Using the results of  \cite{HeinznerSchwarz} we  define a kind of moment mapping on $X$ whose zero set  $\M$ is $H_0:=(H\cap U)$-invariant and has the following properties:
\begin{itemize}
\item An orbit $H*x$ is closed if and only if it intersects $\M$.
\item For every $x\in X$, the orbit closure $\overline{H*x}$ contains a unique $H_0$-orbit in $\M$. 
\item  The inclusion $\M\to X$ induces a homeomorphism $\M/H_0\simeq\quot X{H}$.
\end{itemize}

Now $X_0:=\beta(G_0)\subset G_0$ has the $*$-action of $G_0\supset H_0$.
Let $A$ be a (connected) torus in $X$. We say that $A$ is \st-split if $\sigma(a)=\theta(a)=a\inv$ for all $a\in A$. Now let $A_0$ be a maximal \st-split torus in $G_0$ (so $A_0\subset X_0$). Then it follows from \cite{Matsuki97} (cf.\ Theorem \ref{thm:A0}) that there is a finite Weyl group $W_0^*$ acting on $A_0$ such that the inclusion $A_0\to X_0$ induces a homeomorphism $A_0/W^*_0\simeq X_0/H_0$.     The idea is to try to find a similar result for the $H_0$-action on $\M$.

Let $x\in X$. Then there is a natural submanifold $P_x$ of $X$ which is stable under conjugation by $H_x$ such that $P_xx$ is transversal to the orbit $H*x$ at $x$. If $H*x$ is closed, then an   $H_x$-stable open subset of $P_x   x$ is a slice for the action of $H$. Moreover, $P_x$ is a symmetric space for the  action of $H_x$.  We say that $x$ is a \emph{principal point\/} if the action of $H_x$ on $\SS_x:=T_eP_x$ is trivial. Let $\lieg=\lieg_0\oplus \lier_0$ be the Cartan decomposition of $\lieg$. 
If $u\in G_0$, then   $\SS_u$ is $\delta$-stable and decomposes into a compact part $\SS_u\cap \lieg_0$ and a noncompact part $\SS_u\cap\lier_0$. 
There is a natural  $H_0$-equivariant surjective map $\pi\colon \M\to X_0$ where the fiber of $\pi$ above $u\in X_0$ is $\exp(\SS_u\cap\lier_0)u$.  Thus $\M$ fibers over $X_0$ with fiber over $u$ the noncompact part of the transversal at $u$. We show that there is a natural and finite stratification of $A_0$ which is $W^*_0$-stable such that the mapping $\pi$ is a fiber bundle over each stratum. This in turn implies that   $\quot X{H}\simeq \M/H_0$ is a fiber bundle over the images of the strata in $A_0/W^*_0$. If $u$ lies in a stratum $S$ of $A_0$, then the fiber over the image of $S$ in $A_0/W^*_0$ is $\exp(\SS_u\cap \lier_0)/(H_0)_u$. Moreover, for any maximal $\sigma$-split commutative subspace $\liet$ of  $\SS_u\cap\lier_0$ there is a finite Weyl group $W(S,\liet)$  such that $\exp(\SS\cap \lier_0)/(H_0)_u\simeq\exp(\liet)/W(S,\liet)$.

There is another way to parameterize the quotient  $\quot X{H}$. Let $u\in A_0$ and let $A$ be a $\delta$-stable maximal $(\sigma,\theta_u)$-split torus. Here $\theta_u$ denotes $\theta$ followed by conjugation by $u$.  Then $Au\subset\M$. We say that $A$ (or $Au$) is \emph{standard\/} if $A\cap U=A\cap A_0$. We say that   maximal $(\sigma,\theta_{u_i})$-split tori $A_iu_i$, $i=1$, $2$, are \emph{equivalent\/}  if there is an $h\in H$ such that $h*A_1u_1=A_2u_2$. Then we show the following:
\begin{itemize}
\item For each stratum  of  $A_0/W_0^*$ there is at most one associated standard maximal $Au$. Let $\{A_iu_i\}$ be a maximal  collection of pairwise non-equivalent   tori coming from the strata. Then $\cup_i A_iu_i\to\quot X{H}$ is surjective.    If $x$ is a principal point, then $H*x$ intersects precisely one of the $A_iu_i$ and $P_x$ is a maximal $(\sigma,\theta_x)$-split torus. 
\item If $A_1u_1$  and $A_2u_2$ are standard maximal, then they are equivalent if and only if there is a $w\in W_0^*$ such that $w*(A_1u_1\cap A_0)=A_2u_2\cap A_0$.
\item If $A$  is a  maximal $(\sigma,\theta_u)$-split  torus, then the group of self-equivalences of $Au$ is a finite group. This group acts freely on the set of principal points of $Au$.
 \end{itemize}

 This paper is a natural follow up to \cite{Helm-Schwarz01} where we considered the problem of determining the quotient $G^\sigma\backslash G/G^\theta$ (the complex case, or more generally the case of an algebraically closed field of characteristic not 2).  Our techniques were those of invariant theory, i.e., slice theorems, isotropy type stratifications, etc. We also used these techniques here. The new ingredient is the use of moment-map techniques from \cite{HeinznerSchwarz}.   The moment map techniques have also been used in a recent paper of Miebach \cite{Miebach} who works in the setting of Matsuki's characterization of the double coset spaces.   The main novelty of our results is the use of   stratifications of $A_0/W_0^*$ which help  in determining the topological structure of $\quot X{H}$. Also, our arguments are somewhat simplified since it is only the action of the group $H$ which is being considered. 
     
     \smallskip
     
     We thank the referees for helpful comments.
     
\section{Cartan decomposition, Kempf-Ness set and the quotient}\label{sec:Cartan}

We now show that one can assume the existence of a Cartan involution commuting with $\sigma$ and $\theta$.

\begin{proposition} \label{prop:commute.with.delta} Assume that $\lieg$ is   semisimple. Then there is a   Cartan involution $\delta$ of $\lieg$ such that $\delta$ commutes with $\sigma$ and  $\theta'$ where $\theta'$ differs from $\theta$ by conjugation by an inner automorphism of $\lieg$.
\end{proposition}

\begin{proof}
There are Cartan involutions $\delta$ and $\delta'$ of $\lieg$ which commute with $\sigma$ and $\theta$, respectively. Moreover, $\delta$ and $\delta'$ differ by conjugation by an inner automorphism of $\lieg$ \cite[Ch.\ III, \S 7]{Helgason78}.  
\end{proof}

Let conjugation by $g\in G$ be denoted by $\conj(g)$.

\begin{corollary} For some $g\in G^0$ there
 is a Cartan involution of $G$ commuting with $\sigma$, $\conj(g)\circ\theta\circ\conj(g)\inv$ and the involution $\phi$ defining $G\subset U\c$.
\end{corollary}

\begin{proof} We may write $U\c=(U\c,U\c)Z(U\c)^0$ where $Z(U\c)$ is the center of $U\c$. The two components of the decomposition are preserved by $\sigma$, $\theta$ and $\phi$. Since $Z(U\c)^0$ has a unique maximal compact subgroup $T$, it is preserved by $\sigma$, $\theta$ and $\phi$. From Proposition  \ref{prop:commute.with.delta} we may assume that we have  a Cartan involution $\delta$ of $(\lieg,\lieg)$ which commutes with $\sigma$ and $\theta$. From $\delta$ we obtain a connected maximal compact subgroup $U'$ of $(U\c,U\c)$. Then $U'T$ is a maximal compact subgroup of $U\c$ which is $\sigma$, $\theta$ and $\phi$-stable, hence the corresponding Cartan involution of $G$ commutes with the other involutions.
\end{proof}

\begin{remark}
Changing $\theta$ to a conjugate automorphism as above only changes the double coset spaces we consider by an automorphism \cite[\S 1  Remark 2]{Matsuki97}. Thus from now on we assume that we have a Cartan involution $\delta$ commuting with $\sigma$, $\theta$ and $\phi$ and that $(U\c)^\delta=U$. It is not hard to show that, if $\sigma$ and $\theta$ commute, then one does not need to replace $\theta$ by a conjugate to find a suitable $\delta$.
\end{remark}

 In the following, let $G$, $H$ and $K$ be as in the introduction, where $K=G^\theta$. From $\theta$ we have the eigenspace decomposition $\lieg=\liek\oplus\liep$. Similarly, for $\sigma$ we have $\lieg=\lieh\oplus\lieq$.
  Let $P:=\{x\in G\mid \theta(x)=x\inv\}$. Then $G$ acts on $P$ via $*$; $g$, $x\mapsto g*x:=gx\theta(g\inv)$. We set $\beta(g)=g*e$ for $g\in G$ and we denote $\beta(G)$ by  $P_\theta(G)$ or $X$. Then $\beta$ induces a bijection of $G/K$ onto $X$. We will now see that $X$ is smooth, and it follows easily that $\beta$ is a diffeomorphism. Under the diffeomorphism $\beta$, the left action of $H$ on $G/K$ becomes the $*$-action on $X$.
  
\begin{lemma} \label{lem:Popen}
Every $G$-orbit in $P$ is  open.
\end{lemma}

\begin{proof} Let $x\in P$ and consider  the open set of points   $\exp(v)x$ where $v$ lies in a small neighborhood of $0\in\lieg$. Then $\exp(v)x$ lies in $P$ if and only if $\exp(\theta(v))\theta(x)=x\inv\exp(-v)$ which is equivalent to the condition that $\theta(v)=-(\Ad x\inv)\cdot v$. Now let $w=v/2$. Then 
$$
\exp(w)x\exp(-\theta(w))=\exp(w)\exp((\Ad x)(-\theta(w)))x=\exp(v/2)\exp(v/2)x=\exp(v)x.
$$ 
Thus $G*x$ contains a neighborhood of $x$ in $P$.
\end{proof}

\begin{corollary}  Every $G$-orbit in $P$ is closed and is a smooth submanifold. In particular,  $X$ is a closed submanifold of   $G$.
\end{corollary} 

From our Cartan involution $\delta$ we have the Cartan decomposition $G=G_0\exp\lie r_0$ where $G_0=G\cap U$  and $\lieg = \lieg_0\oplus\lier_0$ is the Cartan decomposition of $\lieg$. The involutions $\sigma$ and $\theta$ preserve  $G_0$, $\lie g_0$ and $\lier_0$.  Set $X_0=P_\theta(G_0)$.
\begin{proposition}
Let $G_0$, etc.\ be as above. Then $X_0=X\cap U$.
\end{proposition}
\begin{proof}
Suppose that $x\in X\cap U$,  $x=\beta(u\exp(v))$ where $u\in G_0$ and $v\in\lie r_0$.  Then $x=u *\beta(\exp(v))$. Thus $u\inv*x=\beta(\exp(v))$, so that we may assume  that $x=\exp(v)\exp(-\theta(v))$ for $v\in\lie r_0$. But then $x\exp(\theta(v))=\exp(v)$ and uniqueness in the Cartan decomposition forces $x=e$. Thus $x\in X_0$.
\end{proof}

The proof of Lemma \ref{lem:Popen} also gives
\begin{lemma}\label{theta:lemma}
Suppose that $x:=u\exp(v)\in X$ where $u\in G_0$ and $v\in\lie r_0$. Then $\theta(u)=u\inv$ and $\theta(v)=-(\Ad u)v$. 
\end{lemma}

\begin{corollary} \label{cor:projection}
There is a $G_0$-equivariant projection $\lambda\colon X\to X_0$, $u\exp(v)\mapsto u$, where $G_0$ is acting via $*$.
\end{corollary}
\begin{proof}
Let $x=u\exp(v)\in X$ as above.   Then $u\in P$. By assumption \eqref{eqn:assumption}, $H_0$ acts transitively on the components  of $X$, so we may assume that there is a path from $x$ to $e$ in $X$, so that there is a path in $P\cap U$ from $u$ to $e$. Since $X_0$ contains the component of $P\cap U$ containing $e$ we have  $u\in X_0$. If $u'\in G_0$, then $u'*x=(u'*u)\exp(\Ad \theta(u')v)\in X$ where $\Ad\theta(u')v\in\lie r_0$. Thus $\lambda$ is equivariant.
\end{proof} 

\begin{remark}\label{rem:anotherXtoX0}
There is another way to view $\lambda\colon X\to X_0$. From \cite[Theorem 9.3]{HeinznerSchwarz} there is an isomorphism $ G_0\times^{K_0}(\liep\cap\lier_0) \simeq G/K$, $[u,\xi]\mapsto u\exp(\xi)$, where $K_0=K\cap U$. This gives us a $G_0$-equivariant map $G/K\to G_0/K_0$. Applying $\beta$ one gets the mapping $\lambda$.\end{remark}
\subsection{Conditions on $\sigma$ and $\theta$}
 
 Our assumption on $\sigma$ and $\theta$ is that they generate a finite group of automorphisms of $Z(U)^0$. This assumption is related to that   of Matsuki \cite[\S 2]{Matsuki97}:
 
 \begin{proposition} \label{prop:sigmatheta}
 The following are equivalent.
 \begin{enumerate}
\item The subgroup of $\Aut(Z(U)^0)$ generated by $\sigma$ and $\theta$ is finite.
\item The image of $\sigma\theta$ in $\Aut(Z(U)^0)\simeq \GL(n,\Z)$ is semisimple with eigenvalues of norm 1.
\end{enumerate}
 \end{proposition}
 
 \begin{proof}
 Clearly (1) implies (2). Now assume (2). The image  of $\sigma\theta$ in $\Aut(Z(U)^0)\simeq\GL(n,\Z)$ has eigenvalues that satisfy a monic polynomial equation with integer coefficients. Since the eigenvalues have norm 1, they must be roots of unity. Since $\sigma\theta$ is semisimple, it follows that $\sigma\theta$ has finite order. Thus   $\sigma$ and $\theta$ generate a finite subgroup of $\Aut(Z(U)^0)$.
 \end{proof}

\begin{lemma}\label{lem:innerproduct}
There is an inner product $\langle\ ,\,\rangle$ on $\lieu$ which is $U$, $\sigma$ and $\theta$-invariant.
\end{lemma}
\begin{proof}
Write $U=U_sZ(U)^0$ where $U_s=(U,U)$ is the semisimple part of $U$. Since $\Aut(U_s)$ is a finite extension of the inner automorphisms $\Int U_s$ of $U_s$, we find that $\sigma$, $\theta$ and $U_s$ generate a compact group of automorphisms of $\lieu_s$, so that we can find an inner product invariant under this compact group. Since $\sigma$ and $\theta$ generate a finite group of automorphisms of $\lie z(\lieu)$, it  has an invariant inner product. Putting the two inner products together gives the desired result.
\end{proof}

 \subsection{Moment mapping and Kempf-Ness set}\label{sec:moment}

We have the Cartan decomposition $U\c\simeq U\times i\lieu\simeq U\times\exp(i\lieu)$. Let $\langle\ ,\,\rangle$ be an inner product on $\lieu$ which is $U$, $\theta$ and $\sigma$-invariant.  We also consider $\langle\ ,\,\rangle$ to be an inner product  on $\lieu^*$.  The function $\rho\colon U\c\to \R$, $\rho(u\exp(i\eta)):=(1/2)\langle\eta,\eta\rangle$, $u\in U$, $\eta\in\lieu$, is a strictly plurisubharmonic exhaustion of $U\c$ (see  \cite[\S 9]{HeinznerSchwarz}) and it is invariant under left and right multiplication by elements of $U$. Associated to $\rho$ we have a K\"ahler form $\omega=2i\partial\bar\partial\rho$ and a moment mapping $\nu\colon U\c\to\lieu^*$. If  $g\in U\c$ and $\xi\in\lieu$,  then $\nu(g)(\xi)=\frac{d}{dt}|_{t=0}\rho(g\exp -it\xi)$.  
Note that $\nu$ is $U$-equivariant, where  $U$ is acting by right multiplication on $U\c$ and  
$u\in U$ acts on $\lieu^*$ by $\alpha\mapsto\alpha\circ\Ad(u\inv)$. Let $\nu^\xi$ denote $\nu$ followed by evaluation at $\xi\in\lieu$.
A computation \cite[Lemma 9.1]{HeinznerSchwarz} shows that $\nu$ has the following simple form: 

\begin{equation}\label{eqn:moment}
\nu^\xi(u\exp(i\eta))=\langle\xi,\eta\rangle,\    u\in U; \ \xi,\  \eta\in \lieu. 
\end{equation}

Thus $\nu$ is, essentially, projection of $U\c\simeq U\times i\lieu\simeq U\times  \lieu^*$ onto the second factor. 
Now we consider the action of $U$ on $U\c$ given by $u*g=ug\theta(u)\inv$. Note that $\rho$ is still $U$-invariant.

\begin{lemma}\label{mulemma}
Let $U$ act on $U\c$ by $*$ and let  $\mu\colon U\c\to \lieu^*$ denote the corresponding moment mapping; $\mu^\xi(g)= \frac{d}{dt}|_{t=0}\rho( \exp(it\xi)g\exp( -it\theta(\xi))$. Then
$$
\mu^\xi(u\exp(i\eta))=\langle  \eta,\theta(\xi)-\Ad(u\inv)(\xi)\rangle,\ u\in U; \ \eta, \ \xi\in\lieu. 
$$
\end{lemma}
\begin{proof}
The product rule shows that 
$$
\mu^\xi(u\exp(i\eta))=\frac{d}{dt}|_{t=0}\rho(\exp(it\xi)u\exp(i\eta))+\frac{d}{dt}|_{t=0}\rho( u\exp(i\eta)\exp(-it\theta(\xi)))
$$
where by \eqref{eqn:moment} the second term is $\langle \eta,\theta(\xi)\rangle$. Now 
$$
\rho(\exp(it\xi) u\exp(i\eta))= \rho(u \exp(it\Ad(u\inv)\xi)\exp(i\eta))=\rho(\exp(it\Ad(u\inv)\xi)\exp(i\eta)).
$$
Let $\zeta$ denote $\Ad(u\inv)\xi$. Write $\exp(i\eta)\exp(-it\zeta)=u(t)\exp(i\gamma(t))$ where $u(t)\in U$ and $\gamma(t)\in \lieu$. Then $\rho(\exp(i\eta)\exp(-it\zeta))=1/2\langle \gamma(t),\gamma(t)\rangle$ and by \eqref{eqn:moment}, $\frac{d}{dt}|_{t=0}\langle \gamma(t),\gamma(t)\rangle=\langle\eta,\zeta\rangle$.  It follows that 
$$
\rho(\exp(it\zeta)\exp(-i\eta))=\rho(\exp(-i\gamma(t))u(t)\inv)=1/2\langle \gamma(t),\gamma(t)\rangle
$$
and that
$$
\frac{d}{dt}|_{t=0}\rho(\exp(it\xi)u\exp(i\eta))=\langle -\eta,\zeta\rangle=\langle \eta,-\Ad(u\inv)\xi\rangle.
$$
\end{proof}
Let $x:=u\exp(i\eta)\in X$ where $u\in  G_0$ and   $\eta\in  i\lier_0$.  
By restriction,  the moment mapping $\mu$ gives us a mapping   $X\to\lieu^*\to(i\lieh\cap \lieu)^*$, which we also denote by $\mu$.   Then $\mu$ is the relevant moment mapping when one considers the action of $H$ on $X$ (see \cite[\S 5]{HeinznerSchwarz} and \S \ref{quotientproperties} below). Let $\M$ denote the Kempf-Ness set, i.e., $\M=\{x\in X\mid  \mu(x)=0\}$. From Lemma \ref{mulemma} we see that 
$$
\M\subset\{u\exp(i\eta)\mid \theta(\eta)-\Ad(u)\eta\perp i\lieh\cap \lieu\}.
$$
From Lemma   \ref{theta:lemma} and Corollary \ref{cor:projection} we have $u\in X_0$ and $\theta(\eta)=(-\Ad u)(\eta)$ so that $2\theta(\eta)\perp i\lieh\cap\lieu$.   Now $i\lieg=i\lieh\oplus i\lieq$ so that the perpendicular to $i\lieh\cap\lieu$ in $i\lieg\cap\lieu=i\lier_0$ is $i\lieq\cap\lieu$. Thus when $u\exp(i\eta)\in\M$ we have that $\theta(\eta)\in i\lieq\cap i\lier_0$.  
We finally get that
\begin{equation}\label{eqn:Mleft}
\M=\{u\exp(\zeta)\mid u\in X_0, \ \theta(\zeta)\in \lieq\cap\lier_0 \text { and } \theta(\zeta)=-\Ad(u)(\zeta)\}.
\end{equation}

 \subsection{The quotient}\label{quotientproperties}
 
 Recall that $H_0=H\cap U$. Since the moment mapping comes from a strictly plurisubharmonic exhaustion, from \cite[10.2, 11.2, 11.15, 13.4]{HeinznerSchwarz} we have the following result 
 
 \begin{theorem}\label{thm:orbitclosure} Let  $x\in X$. Then the orbit closure $\overline{H*x}$ contains a point $x_0$ of $\M$. Moreover, $\overline{H*x}\cap \M=H_0x_0$ is a single $H_0$-orbit. 
 \end{theorem}
 
 We  now define an equivalence relation on $X$ by 
$x\sim y$ if the corresponding $H_0$-orbits in $\M$ coincide. We define the quotient $\quot X{H}$ to be  the set of all equivalence classes with the quotient topology.  From \cite[11.2, 13.5]{HeinznerSchwarz} we have

\begin{theorem}\label{thm:homeom} The inclusion $\M\to X$ induces a homeomorphism $\M/H_0\simeq \quot X{H}$. The closed $H$-orbits in $X$ are precisely those which intersect $\M$.  
\end{theorem}

\section{Transversals and the slice theorem}\label{sec:transversals}

Let $x\in X$. Let $\theta_x$ denote $\conj(x)\circ\theta$ and let $\tau_x$ denote $\theta_x\sigma$. Then $\theta_x$ is an involution of $G$. Let $G^{(x)}$ denote the fixed points of $\tau_x$. Then $\conj(x\inv)$ and $\tau=\theta\sigma$ are the same on $G^{(x)}$ as are $\sigma$ and $\theta_x$. Moreover, $G^{(x)}$ is $\sigma$ and $\theta_x$-stable. Note that $G^{(x)}$ is not necessarily reductive (see Corollary \ref{cor:semisimple} below).
Let $\SS_x$ denote $\{Z\in\lieg\mid\theta_x(Z)=\sigma(Z)=-Z\}$. Let $P_x$ denote the component of $P_\theta(G^{(x)})$ containing $e$. Let $\liep^{(x)}$ denote $\{Z\in\lieg\mid \theta_x(Z)=-Z\}$ and let $\liek^{(x)}$ denote $\{Z\in\lieg\mid\theta_x(Z)=Z\}$. Then $\lieg=\liek^{(x)}\oplus\liep^{(x)}$. The subgroup of $G^{(x)}$ fixed by $\sigma$ is just $H_x$, the isotropy group of $H$ at $x$.  

We say that a locally closed $H_x$-stable smooth subset $Q_x$ of $X$ is a \emph{transversal\/} at $x$ if   $T_x(Q_x)$ is a complement to $T_x(H*x)$ in $T_xX$. The following theorem generalizes the results in \cite{Helm-Schwarz01}.  
 
\begin{theorem}\label{thm:generalslice} Let $x\in X$. Then $P_xx$ is a transversal to $H*x$. Moreover, $T_e(P_x)=\SS_x$ and $T_e(H_x)=\liek^{(x)}\cap\lieh$.
\end{theorem}

\begin{proof} We have the smooth map $\psi\colon G\to Xx\inv$, $g\mapsto (g*x)x\inv$. The differential $d\psi$ of $\psi$ at $e$ sends $Z\in\lieg$ to $Z-\theta_xZ$. The square of $d\psi$ is $2d\psi$, so that $d\psi$ is (up to a constant) a projection of $\lieg$ onto $T_e(Xx\inv)$. Now $d\psi(\lieh)$ is the tangent space to $(H*x)x\inv$ at $e$.  
The kernel of $d\psi$ is $\liek^{(x)}$, so that $\liep^{(x)}$ is complementary to the kernel of $d\psi$.  In fact, $d\psi$ is multiplication by $2$ on $\liep^{(x)}$. Thus the Lie algebra of $H_x$ is $\lieh\cap\liek^{(x)}$ and the complement to $T_e((H*x)x\inv)$ in $T_e(Xx\inv)$ is $d\psi(\lieq\cap\liep^{(x)})=\lieq\cap\liep^{(x)}=\SS_x$. Now $T_eP_x=\{Z\in\lieg^{(x)}\mid\sigma(Z)=-Z\}=\{Z\in\lieq\mid \theta_x(Z)=-Z\}=\SS_x$. It follows that $P_xx$ is a transversal to $H*x$ at $x$.
 \end{proof}
 
 From \cite[14.10] {HeinznerSchwarz} (see also \cite[2.7]{Helm-Schwarz01}) we obtain the slice theorem.
 \begin{corollary}\label{slicecorollary}
Suppose that $H*x$ is closed. Then there is an open $H_x$-stable subset $S\ni x$ of $P_xx$ which is a slice at $x$ for the action of $H$ on $X$. In other words, the canonical mapping $H\times^{H_x}S\to X$, $[h,s]\mapsto h*s$, is an $H$-equivariant diffeomorphism onto an open subset of $X$.
 \end{corollary}
 
 We now compare the transversals along $H$-orbits.
 
\begin{proposition}\label{slicemovecorollary} Let $x\in X$   and let $h\in H$.  Then $\conj(h) (G^{(x)})=G^{(h*x)}$ and $\conj(h) (P_x)=P_{h*x}$. Thus $h*(P_xx)=\conj(h)(P_x)h*x=P_{h*x}h*x$, hence $h$ carries the transversal at $x$ to the transversal at $h*x$.
\end{proposition}

\begin{proof} Let $g\in G^{(x)}$, i.e., assume that $\tau_x(g)=g$. Then  
$$
\tau_{h*x}(\conj(h) g)=hx\theta(h)\inv\theta(h)\tau(g)\theta(h)\inv\theta(h)x\inv h\inv=hx\tau(g)x\inv h\inv=\conj(h) g.
$$
Thus $\conj(h) G^{(x)}\subset G^{(h*x)}$ and similarly $\conj(h\inv)  G^{(h*x)}\subset G^{(x)}$, so we have equality.
It follows that $\{g\sigma(g)\inv\mid g\in G^{(x)}\}$ is sent by $\conj(h)$ onto $\{g\sigma(g\inv)\mid g\in G^{(h*x)}\}$, so $h$ maps $P_xx$ isomorphically onto $P_{h*x}(h*x)$.
\end{proof}

By a torus in $X\subset G$ we mean a connected commutative group of semisimple elements. Thus $\R^*$ is not a torus for us, but its identity component is. Now as above, we have the following

 \begin{proposition}\label{hproposition}
 Let $x\in X$, let $h\in H$ and let $A$ be a $(\sigma,\theta_x)$-split torus in $G$. Then $\conj(h) A$ is $(\sigma,\theta_{h*x})$-split and $h*(Ax)=(\conj(h)  A)h*x$.
 \end{proposition}
   
 \begin{definition}
 We say that $x\in X$ is a \emph{principal point\/} if $H*x$ is closed and the slice representation at $x$ is trivial, i.e., $H_x$ acts trivially on $\SS_x$. We say that $x\in \M$ is principal if it is principal in $X$.
 \end{definition} 
 
 \begin{remark} \label{rem:principal}
 If $x$ is principal, then the slice representation at $x$ must have dimension $\dim \quot X{H}$ which is the same as the dimension of a maximal \st-split torus in $X$.   But this is the same as the dimension of a maximal $(\sigma,\theta_x)$-split torus in $P_x$, hence   $P_x$  must be a maximal $(\sigma,\theta_x)$-split torus. Finally, it is well-known that the principal points for the $H_x$-action on $P_x$ are open and dense, hence the same is true for $H$ acting on $X$.   
 \end{remark}
  
  \section{The compact case}\label{sec:compact}
We first need to compute the quotient of $X_0$ by the action of $H_0$. This is in  \cite[\S 3]{Matsuki97}, but we need Corollary \ref{cor:weylK0'}. A variant of an argument in  \cite[6.7]{Helm-Schwarz01}  gives this result. 

Let $A_0$ be a fixed maximal \st-split torus in $X_0$. 
Set $W_0^*=N_0^*/Z_0^*$ where $N_0^*=\{h\in H_0\mid h*A_0=A_0\}$ and $Z_0^*=\{h\in H_0\mid h*a=a$ for all $a\in A_0\}$.   We have the following elementary lemma (c.f. \cite[1.10, 4.5]{Helm-Schwarz01})

\begin{lemma}\label{lem:weylgroup}
\begin{enumerate}
\item $N_0^*=\{h\in N_{H_0}(A_0)\mid \beta(h)\in A_0\}$.
\item For $h\in H_0$, $\beta(h)\in A_0$ if and only if there is an $s\in A_0$ such that $s\inv h\in K_0$, in which case $s^2=\beta(h)$.
\item For $h\in  H_0$, $h\in N_0^*$ if and only if $hA_0K_0=A_0K_0$ and $h\in Z_0^*$ if and only if $haK_0=aK_0$ for all $a\in A_0$.
\item If $h\in N_0^*$, let $s\in A_0$ such that $hK_0=sK_0$. Then $haK_0=(h*a)s\inv K_0=hah\inv sK_0$ for all $a\in A_0$.
\end{enumerate}
\end{lemma}

\begin{remark} \label{rem:w0*finite}
It follows from Theorem \ref{thm:generalslice} that the intersection of $H_0*e$ and $A_0$ is discrete, hence finite. Thus for any $h\in N_0^*$, there are only finitely many possibilities for $\beta(h)$. Hence $W_0^*$ is finite since  $N_{H_0}(A_0)/Z_{H_0}(A_0)$ is finite.
\end{remark}
 
\begin{proposition}\label{prop:w*Y}
Let $h\in H_0$ such that $h*Y= Y'$ where $Y$, $Y'\subset A_0$ are nonempty. Then there is a $w\in W_0^*$ such that $w*y=h*y$ for all $y\in Y$.
\end{proposition}

\begin{proof}
Let $G_0^Y=\{g\in G_0\mid \tau_y(g)=g$ for all $y\in Y\}$ and define $G_0^{Y'}$ similarly. Then a calculation shows that $\conj(h)G_0^Y=G_0^{Y'}$.  Now $A_0$ is maximal $\sigma$-split in both $G_0^Y$ and $G_0^{Y'}$, and $\conj(h\inv) A_0$ is maximal $\sigma$-split in $G_0^Y$. Thus there is an element $g$ in the identity component of $(G_0^Y)^\sigma$  such that $g\inv h\inv A_0=A_0$. Since $g\in G_0^Y$ and $\sigma(g)=g$, it follows that for all $y\in Y$ we have $\theta_y(g)=g$ and hence $g*y=y$ . Thus $hg\in N_{H_0}(A_0)$ and $hg$ acts on $Y$ in the same way as $h$. Since $hg*y\in A_0$ and $\conj(hg)y\in A_0$ for $y\in Y$, it follows that $hg\in N_0^*$. Hence $hg$ gives the requisite element of $W_0^*$.
\end{proof}

 Now we consider the case where $K$ is replaced by an open subgroup. Let $K'$ denote an open subgroup of $K$, let $K_0'$ denote $K'\cap U$ and  let $X_0'$ denote $G_0/K_0'$.  We have the $H_0$-equivariant covering map $\beta\colon X'_0\to X_0$. Let $N_0'=\{h\in N_{H_0}(A_0)\mid hA_0K_0'=A_0K_0'\}$, let $Z_0'=\{h\in H_0\mid haK_0'=aK_0'$ for all $a\in A_0\}$ and define $W_0'=N_0'/Z_0'$. Let $A_0^{(2)}$ denote  the elements of $A_0$ of order 2. Then $A_0^{(2)}=A_0\cap K$. There is a natural morphism $W_0'\to W_0^*$ and  we have

\begin{lemma}\label{lem:wprime}
The canonical mapping $W_0'\to W_0^*$ has kernel $A_0^{(2)}K_0'\cap Z_{H_0}(A_0)K_0'$. The image is represented by $\{h\in N_0^*\mid s\inv h\in K_0'$ for some $s\in A_0\}$.
\end{lemma}

\begin{corollary} \label{cor:weylK0'} Let $Y$, $Y'$ be nonempty subsets of $A_0$ and suppose that $h\in H_0$ such that $hY(K_0')=Y'(K_0')$. Then there is a $w\in W_0'$ such that $hy(K_0')=wy(K_0')$ for all $y\in Y$.
\end{corollary}

\begin{proof}
As in the proof of Proposition \ref{prop:w*Y} we can find $g$ in the identity component of $M:=(G^Y)^\sigma$ such that $hg\in N_0^*$.  We saw that    $YK_0\subset X_0^M$, hence the elements of the Lie algebra $\lie m$, considered as vector fields on $X_0$, vanish on $YK_0$. It follows that the elements of $\liem$, considered as vector fields on $X_0'$, vanish on $Y(K_0')$. Thus  $g$ fixes $Y(K_0')$. Hence we may reduce to the case that $h\in N_0^*$.  For $y\in Y$ we have $hy(K_0')=y'(K_0')$ where $y'\in A_0$. It follows that $(y')\inv hyh\inv h\in  K_0'$ where $s:=(y')\inv hyh\inv \in A_0$. Thus  $h$ induces an element of $W_0'$.
\end{proof}

From  \cite[\S 3]{Matsuki97} we have 
\begin{theorem}\label{thm:A0} 
\begin{enumerate}
\item $G_0=H_0A_0K_0'$.
\item  The inclusion $A_0K_0'\subset X_0'$ induces a homeomorphism $(A_0K_0')/W_0'\simeq X_0'/H_0$.
\end{enumerate}
\end{theorem} 

Note that Corollary \ref{cor:weylK0'} is a strengthening of (2) above.

  \section{Parameterization of the quotient}
 We find it useful to consider the Kempf-Ness set $\M$ with the orders of the compact and noncompact parts reversed. Using \eqref{eqn:Mleft} one easily shows
 
 \begin{proposition} \label{prop:M} We have 
 $$
 \M=\{\exp(\xi)u\mid u\in X_0, \ \xi\in\SS_u\cap\lier_0\}.
 $$
  \end{proposition}

\begin{remark}\label{rem:fibers}
The description of $\M$ shows that $\M$ fibers over $X_0$ with fiber  over $u\in X_0$ the noncompact part $\exp(\SS_u\cap\lier_0)u$ of the transversal $P_uu$.
\end{remark}

\begin{corollary}\label{cor:semisimple}
Let $x\in X$ such that $H*x$ is closed. Then $\tau_x=\theta_x\sigma$ is semisimple so that $G^{(x)}$ is reductive.
\end{corollary}

\begin{proof} First assume that $x\in\M$, so that $x= \exp(\xi)u$ for $u\in X_0$ and $\xi\in\SS_u\cap\lier_0$. It follows from Lemma \ref{lem:innerproduct} that $\tau_u$ is semisimple, and since $\xi\in\SS_u$, $\tau_u(\xi)=\xi$. Since $i\xi\in\lieu$, $\ad i\xi$ and $\ad\xi$ are semisimple endomorphisms of $\lieu\c$, hence $\conj(\exp(\xi))$ is a semisimple automorphism of $U\c$ commuting with $\tau_u$.  Thus $\tau_x=\conj(\exp(\xi))\tau_u$ is semisimple.

Now suppose that $h\in H$ and $x\in\M$. Then 
one computes that $\tau_{h*x}=\conj(h)\tau_x\conj(h\inv)$, so $\tau_{h*x}$ is semisimple.
\end{proof}

 \begin{lemma}\label{lem:Au} Let $u\in A_0$ and let $A$ be a $\delta$-stable $(\sigma,\theta_u)$-split torus. Then $Au\subset\M$.
 \end{lemma}
 
 \begin{proof}
 We may write $A=BC$ where $B$ is a $\delta$-split torus and $C$ is $\delta$-fixed, both tori being $\sigma$-split. Let $b=\exp(Z)\in B$ and $c\in C$. Since $bc$ is $\theta_u$-split, it follows that both $b$ and $c$ are $\theta_u$-split, so that $Z$ is $\theta_{cu}$-split. Hence $Z\in\SS_{cu}\cap \lier_0$ and $bcu\in\M$.
 \end{proof}
 
 \begin{remark} If $Z\in\lie r_0$, then $\ad Z$ acts on $\lieg$ with real eigenvalues, hence an element of $\lieg$ is fixed  by $\Ad(\exp(Z))$ if and only if it is annihilated by $\ad Z$.
\end{remark}

 We now give a small subset of $\M$ mapping onto the quotient $\quot X{H}$.
 
 \begin{theorem}\label{thm:main}
 There are points $u_1,\dots,u_s\in  A_0$ and   maximal $(\sigma,\theta_{u_i})$-split $\delta$-stable tori $A_i$ 
 such that every closed $H$-orbit in $X$ intersects $\cup A_iu_i$. Moreover, every principal orbit intersects exactly one of the $A_iu_i$.
 \end{theorem}
 
 \begin{proof}
 Let $x\in\M$, so that $x=\exp(Z)u$ where $u\in X_0$ and $Z\in\SS_u\cap\lie r_0$. Theorem \ref{thm:A0} allows us to assume that   $u\in A_0$. 
Now the transversal at $u$ is $P_uu$ where $P_u$ is the symmetric space associated to $G^{(u)}$ and $\sigma$. We have the transversal $Q_xx$ at $x$ for the action of $H_u$ on $P_uu$. Then 
\begin{align*}
T_e(Q_x)&=\{Y\in \lieg^{(u)}\mid\sigma(Y)=\theta_{x}(Y)=-Y\} \\&=\{Y\in\SS_u\mid \Ad(\exp(Z))Y=Y\}=\{Y\in \SS_u\mid [Z,Y]=0\}.
\end{align*}
 If $x$ is principal, then $\exp(Z)$ is principal for the action of $H_u$ on $P_u$, hence  $Q_x$ has to be a maximal ($\sigma,\theta_u$)-split  torus $A$ in $G$. By construction,  $A$ is $\delta$-stable, hence $Au\subset\M$. Now we have $\exp(Z)\in A$ and $x\in Au$. We have shown that the principal orbits intersect a union $\cup A_iu_i$, but we do not yet know that finitely many $A_iu_i$ suffice.
 
 Now suppose that we have $A_1u_1$ and $A_2u_2$ where $u_1$, $u_2\in A_0$ and the $A_i$   are $\delta$-stable maximal $(\sigma,\theta_{u_i})$-split tori in $X$, $i=1$, $2$. Suppose that there are principal points  $x_i:=a_iu_i$ and an $h\in H_0$ such that $h*x_1=x_2$.  Since $x_1$ is principal,  $P_{x_1}$   is a maximal $(\sigma,\theta_{x_1})$-split torus. But $A_1$ is $\sigma$-split and $\theta_{x_1}$-split (since $a_1\in A_1$), so  $P_{x_1}=A_1$ and for the same reason, $P_{x_2}=A_2$. Then Proposition \ref{hproposition} shows that $\conj(h) A_1=A_2$ and hence $h*A_1u_1=A_2u_2$.
 Thus the principal points of $\M/H_0$ are the image of the disjoint union of the principal points of sets $A_iu_i$. Moreover, the principal points of the $A_iu_i$ have open image in $\M/H_0$, so that each irreducible component of the principal points of $\M/H_0$ lies in the image of a single $A_iu_i$. Now the principal points of $\M/H_0$ are a semialgebraic subset of $\M/H_0$, so that there are   finitely many components. Hence we only need a finite number of $A_iu_i$.  Since the set of principal orbits is open and dense in $X$, the same is true for $\M$. It follows that $H_0*\cup_i A_iu_i=\M$.   Hence every closed $H$-orbit in $X$ contains a point of $\cup_i A_iu_i$. 
 \end{proof}

 \begin{definition}
 Let $A_i$ be maximal $(\sigma,\theta_{u_i})$-split tori where $u_i\in A_0$, $i=1$, $2$. We say that $A_1u_1$ and $A_2u_2$ are \emph{equivalent\/}  if there is an $h\in H$ such that $h*A_1u_1=A_2u_2$.  Equivalently, $\conj(h) A_1=A_2$ and $h*u_1\in A_2u_2$. We define the Weyl group $W^*_{H}(A_1u_1)$ to be $N^*_{H}(A_1u_1)/Z^*_{H}(A_1u_1)$ where the $*$ just reinforces the fact that $H$ is acting via $*$ and not by conjugation.
 \end{definition}
 
 \begin{remark}\label{rem:usualWeyl}
 If $u=e$, then $W^*_{H}(Au)=W^*_{H}(A)$ is the usual (twisted) Weyl group of $A$. 
\end{remark}
 
We saw above that whenever $A_1u_1$ and $A_2u_2$ have images in $\quot X{H}$ with a common principal point, then they are equivalent by an element of $H_0$. The self equivalences of $A_1u_1$ are just $W^*_{H}(A_1u_1)$.

  \begin{example}
  \label{sl2example:part1}
   Here $U\c=\SL(2,\C)$, $U=\SU(2,\C)$   and $G=\SL_2(\R)$. We have the involution $\theta$ which is conjugation with $\sm 0  1  1   0$ and the involution $\sigma$ which is conjugation with $\sm 1 0 0 {-1}$. Then $\tau=\theta\sigma$ is conjugation with $v:=\sm 0 {-1}1     0$ and $\sigma\theta$ is conjugation with $-v$, so that  $\sigma$ and $\theta$ commute. The group $H$ is $\{\sm \lambda 0 0 {\lambda\inv}\mid \lambda\in\R^*\}$, $\lieq=\{\sm 0bc0\colon b$, $c\in\R\} $, $K=\{\sm a b b a\mid a,\ b\in\R,\ a^2-b^2=1\}$ and $\liep=\{\sm ab{-b}{-a}\mid a$, $b\in\R\}$. The Cartan involution $\delta$ of $U\c$ is conjugate inverse transpose,   $\delta$ commutes with $\sigma$ and $\theta$ and $\lier_0=\{\sm abb{-a}\mid a$, $b\in\R\}$. We have $G_0=G\cap U=\SO(2,\R)=A_0=X_0$ and $X:=P_\theta(G)=\{\sm a b {-b} d\mid ad+b^2=1\}$.  The element $\sm \lambda 0 0 {\lambda\inv}\in H$ sends $\sm a b {-b} d$ to $\sm{\lambda^2a} b{-b}{\lambda^{-2}d}$. Thus $H_0=\{\pm I\}$ acts trivially on $X$ so that $\quot X{H}\simeq \M$ and the orbits of $H$ are connected.   
   	
	Use coordinates $x=b$, $y=(a+d)/2$ and $z=(a-d)/2$ on $X$. In these coordinates,  $X$ is just the hyperboloid in 3-space given by the equation $x^2+y^2=1+z^2$. The action of $H$ fixes $x$ and on $y$ and $z$ it is given by matrices $\{\sm a b b a\mid a^2-b^2=1$, $a>0\}$. Now consider the intersection $X\cap\{x=c\}$ for $c$ fixed.
\begin{enumerate}
\item If $c^2\neq 1$, then $X\cap\{x=c\}$ is the hyperbola $y^2-z^2=1-c^2$, each of whose branches is an $H$-orbit. If $1-c^2>0$ (resp.\ $1-c^2<0$), then the $H$-orbit contains a unique  point where $z=0$ (resp.\ $y=0$).  
\item If $c^2=1$, then $X\cap \{x=c\}$ is the union of the two lines $\{(x,y,z)=(c,s,s)\}$ and $\{(x,y,z)=(c,s,-s)\}$ where $s\in \R$. There are two $H$-fixed points $(\pm 1,0,0)$ and eight non-closed $H$-orbits $\{(\pm 1,\pm s,\pm s)\mid s>0\}$.
\end{enumerate}

   Set $M:=\{(x,y,z)\in X\mid y=0$ or $z=0\}$ which is the union of $A_0=X_0$ and the hyperbola $\{x^2-z^2=1$, $y=0\}$. From the above,   each closed $H$-orbit intersects $M$ in a unique point.  We show that $M=\M$.
      
   Let $u\in A_0$. From Remark \ref{rem:fibers} the fiber of $\M$ over $u$ is $\exp(\SS_u\cap\lier_0)u$.  If $\sm 0 ab0\in\lieq$ lies in $\SS_u$, then it is fixed by $\tau_u$ which is conjugation with $uv$. For $u\neq\pm v$, conjugation by $uv$ fixes only the Lie algebra $\{\sm 0a{-a}0\}$ of $A_0$. Thus $\SS_u\cap \lier_0=\{0\}$ and the fiber of $\M$ over $u$ is just the point  $u$. Now suppose that $u=\pm v$. Then $uv$ acts trivially on $\lieq$, so that $\SS_u\cap\lier_0=\{\sm 0aa0\}$ whose exponential is the torus $A:=K^0=\{\sm abba\mid a^2-b^2=1 \text { and } a>0\}$. Thus the fiber of $\M$ above each point $\pm v$ is the translated torus $A(\pm v)$ which, viewed in our coordinates on $X$, is the branch of $\{x^2-z^2=1$, $y=0\}$ through $v$. Hence $M=\M$.  Note that the principal points in $\M$ are just the complement of $\pm v$. The (translated) tori of Theorem \ref{thm:main} are $A_0$, $Av$ and $A(-v)$, each with trivial Weyl group.
\end{example}

\subsection{Non-closed orbits} It is possible to classify all the orbits of $H$ on $X$, not just the closed orbits. Let $H*y$ be a non-closed orbit, and let $x$ be a closed orbit in the closure of $H*y$ (see Theorems \ref{thm:orbitclosure} and \ref{thm:homeom}). We may assume that $x\in\M$. Then $H*y$ has to intersect the transversal $P_xx$, so we may assume that $y\in P_xx$. Recall that $G^{(x)}$ is reductive (Corollary \ref{cor:semisimple}). Now it is well-known that the non-closed orbits for the action of $H_x$ on $P_x$ are the nontrivial unipotent orbits, i.e., those consisting of unipotent elements of $G^{(x)}$  \cite{Rich82a}. But the unipotent elements are those of the form $\exp(n)$ where $n\in\SS_x$ is a nilpotent element of the semisimple part of $\lieg^{(x)}$. There are only finitely many $H_x$-conjugacy classes of such nilpotent elements, so one is reduced to a computation involving symmetric space representations. Hence one can compute all the $H$-orbits on $X$. In Example \ref{sl2example:part1} above, the only points $x\in\M$ where the slice representation of $H_x$ is nontrivial are the non-principal points $v$ and $-v$. At $v$, we have $G^{(v)}=G$ and $\theta_v=\sigma$. Thus $\SS_v=\{\sm 0 a b 0\mid a,\ b\in\R\}$. The representation of $H_v=H\simeq\R^*$ on $\SS_v$ has weights $2$ and $-2$ of multiplicity  one, and the corresponding weight spaces are generated by $n_2:=\sm 0 1 0 0 $ and $n_{-2}:=\sm 0 0 1 0$, respectively. There are two nontrivial $H$-orbits in each of $\R\cdot n_{\pm 2}$. Thus there are four non-closed orbits in $X$ with closure containing $v$, as we saw above. Similarly, there are four non-closed orbits whose closures contain $-v$.

\section{Some results on  Lie algebras}

In this section only, $G$ will denote a general real reductive Lie group with Cartan involution $\delta$ commuting with an involution $\sigma$. We have the Cartan decomposition $\lieg=\lieg_0\oplus\lier_0$ where $G_0=G^\delta$ and we have the decomposition $\lieg=\lieh\oplus\lieq$ relative to $\sigma$. The groups $H$ and $H_0$ are defined as before.

   \begin{proposition}\label{prop:a}  Let $\lieb_0\subset\lieg_0$ be a  maximal $\sigma$-split commutative subalgebra.   Let $\lieb_1\subset \lier_0\cap\lieq$ be maximal abelian in the centralizer of $\lieb_0$. Then
      \begin{enumerate}
      \item $\lieb:=\lieb_0\oplus\lieb_1$ is maximal abelian in $\lieq$.
\item  Let $\lieb_1'$ be another choice of $\lieb_1$. Then $\lieb_1$ and $\lieb_1'$ are conjugate by an element of the connected centralizer of  $\lieb_0$ in $H_0$.
    \end{enumerate}
    \end{proposition}
      
   \begin{proof} Part (1) is obvious. For (2) 
  replace $G$ by the connected centralizer of $\lieb_0$ and then divide by the center. This reduces us to the case that 
$ \lieg_0\cap\lieq=\{0\}$, so that    $\lieg_0\subset \lieh$ and $G_0=H_0$.  Moreover, $\lieb_1$ and $\lieb_1'$ are maximal abelian subspaces of $\lieq\subset\lier_0$. Now $\lieb_1$  and $\lieb_1'$ can be extended to   maximal abelian subspaces of $\lier_0$. But all such subspaces are $(G_0)^0$-conjugate. Since $G_0=H_0$ preserves $\lieq$, we see that $\lieb_1$ and $\lieb_1'$ are $(H_0)^0$-conjugate.
   \end{proof}
   
   \begin{corollary}\label{cor:bconja}
   \begin{enumerate}
\item   Let $\lieb_1$ be as above and set $B_0=\exp(\lieb_0)$. Then $B:=B_0\exp(\lieb_1)$ is a  maximal   $\sigma$-split torus in $G$.
\item Let $B'$ be a maximal $\sigma$-split torus in $G$ such that the maximal compact subgroup of $B'$ has the same dimension as $B_0$. Then $B'$ is $(H_0)^0$-conjugate to $B$.
\end{enumerate}
\end{corollary}

\begin{proof}
Part (1) is immediate from Proposition \ref{prop:a}. As for (2), we may assume that $B'$ is $\delta$-stable \cite[Remark after Lemma 5]{Oshima-Matsuki80}. Then we may decompose $B'$ as $(B'\cap U)\times(B'\cap R_0)$ where $R_0=\exp(\lier_0)$. Since $B'\cap U$ has the same dimension as $B_0$, we know that $B'\cap U$ and $B_0$ are $(H_0)^0$-conjugate. Thus we may assume that $B'\cap U=B_0$. By  Proposition \ref{prop:a} there is an element of $(H_0)^0$ which centralizes $B_0$ and conjugates $B'\cap R_0$ into $B\cap R_0$.
\end{proof}

 If $\lie m$ and $ \lie n$ are subalgebras of $\lieg$, we denote by $\liem^\lien$ the centralizer of $\lien$ in $\liem$. 

\begin{lemma} \label{centerlemma}   The subalgebra $\lieq^\lieh$ lies in the center of $\lieg$.
\end{lemma}

\begin{proof} We may reduce to the case that $\lieg$ is semisimple.   
Restricting $\sigma$ to $\lieg^\lieh$ we obtain that  $\lieg^\lieh=(\lieh^\lieh=Z(\lieh))\oplus\lieq^\lieh$. If $\liea$ is a maximal   commutative subalgebra of $\lieq^\lieh$, then $\exp(Z(\lieh))$ applied to $\liea$ generates the vector space $\lieq^\lieh$. Since the action of $\lieh$ is trivial on $\liea$, we must have that $\liea=\lieq^\lieh$, i.e., $\lieq^\lieh$ is abelian. Since $\lieh$ is reductive, we have an $\lieh$-module decomposition $\lieq=\lieq^\lieh\oplus \lieq'$ for some $\lieh$-module $\lieq'$.   Let $\liea$ be a maximal abelian subspace of $\lieq$ containing $\lieq^\lieh$. Then $\liea=\lieq ^\lieh\oplus\liea'$ where $\liea'\subset \lieq'$. Again we have the fact that $\exp(\lieh)\cdot \liea'$ has to generate $\lieq' $. But then it follows that every element of $\lieq' $ commutes with $\lieq ^\lieh$. Thus $\lieq ^\lieh$ centralizes $\lieq $. Since $\lieh$ obviously commutes with $\lieq^\lieh$, we have shown that $\lieq^\lieh\subset Z(\lieg)$. 
\end{proof}

\section{Stratification and Weyl groups}  
 We have the stratification of $A_0$ coming from the conjugacy class of the $H$-isotropy group. So $a$ and $b$ are in the same stratum if and only if $H_a$ is conjugate to $H_b$. Another way to think of this is to map $A_0$ into $\quot X{H}$ and pull back the stratification of $\quot X{H}$ by isotropy type. We refine this stratification by taking connected components of the strata.   Since the isotropy type stratification is locally finite, it is finite when restricted to $A_0$.
 
 We will show that if $S$ is a stratum, then the translated tori $Au$, where $A$ is maximal ($\sigma$, $\theta_u$)-split and $u\in S$, are independent of $u$. Moreover,   if $\bar S_1\cap S_2\neq\emptyset$ where  the $S_i$ are strata, then the tori for $u\in S_1$ are included in those for $u\in S_2$. Thus if one wants to find the maximal tori of Theorem \ref{thm:main}, one needs to  only consider those $S_i$ which are minimal, i.e., closed. On the other hand, we show that to every stratum $S$ there is associated at most one $Au$, and a subset of these tori gives a collection as in Theorem \ref{thm:main}.  Later we will see that we can determine the  strata of $A_0$ using a Weyl group.

  If $u\in A_0$, let $A_{0,u}$ denote $(A_0^{\conj(H_u)})^0$.   The   stratum $X^{(H_u)}$ of $X$ corresponding to $H_u$ is $\{x\in X\mid H_{x_0}$ is conjugate to $H_u$ where $H*x_0\subset\overline{H*x}$ is closed$\}$.  Restricted to $A_0$ the stratum is just those points with isotropy group conjugate to $H_u$.

\begin{lemma}\label{lem:fixedpoints} Let $u\in A_0$. Then $X^{(H_u)}\cap A_0$ is contained in a finite union $\cup_i A_{0,u_i}u_i$. 
\end{lemma}

\begin{proof}
The fixed point set $A_0^{H_u}$ for the $*$-action is just $A_0^{\conj(H_u)}u$.  If a point $u'\in A_0$ is $H$-conjugate to a point of $A_0^{\conj(H_u)}u$, then it is $W_0^*$-conjugate, so that we get a finite union covering $X^{(H_u)}\cap A_0$ as claimed.
\end{proof}

Since only finitely many strata intersect $A_0$ we have 

\begin{corollary}\label{cor:Sprops}
There is a finite $W_0^*$-stable collection $\{ A_{0,u_i}u_i\}$ with the following properties.
\begin{enumerate}
\item For $i\neq j$,  $A_{0,u_i}u_i\neq  A_{0,u_j}u_j$.
\item If $u\in A_0$, then $u\in A_{0,u_i}u_i$ for some $i$ where $H_{u_i}=H_u$.
\end{enumerate}
\end{corollary}

Now for a fixed $i$, let $V_i$ denote the union of the $A_{0,u_j}u_j$ such that $H_{u_i}$ is properly included in $H_{u_j}$. Then $T_i:=A_{0,u_i}u_i\setminus V_i$ is Zariski open and dense in $A_{0,u_i}u_i$ and is the set of points of $A_{0,u_i}u_i$ whose isotropy group is   $H_{u_i}$. Now define the strata of $A_0$ to be the connected components of the $T_i$.

\begin{proposition}\label{prop:Slocal}
Let $u\in A_0$ and let $S$ be the stratum containing $u$. Then  for $u'\in S$ we have $H_u=H_{u'}$, $P_u=P_{u'}$ and $P_uu=P_{u'}u'$.   
 \end{proposition}
 
 \begin{proof} We have $u'=au$ for some $a\in A_{0,u}$. By definition,  $H_u=H_{u'}$. Now the slice representation of $H_u$  is its action on $\SS_u=\lieg^{(u)}\cap\lieq$. Moreover, $\lieg^{(u)}\cap\lieh$ is the Lie algebra of $H_u$. Since $P_u$ is determined by $\lieg^{(u)}$, we see that $P_u$ is determined by the slice representation   at $u$. But slice representations are constant, up to isomorphism, along connected components of isotropy strata.   By Lemma     \ref{centerlemma} applied to the action of $\sigma$ on $\lieg^{(u)}$ we obtain that  $(\lieg^{(u)})^{H_u}\cap\lieq$ lies in the center of $\lieg^{(u)}$.  It follows that $A_{0,u}$ centralizes   $\lieg^{(u)}$ so that $\lieg^{(u')}\supset\lieg^{(u)}$. Thus 
 $\lieg^{(u')}=\lieg^{(u)}$ and  $P_u=P_{u'}$. Since  $A_{0,u}$ centralizes $\lieg^{(u)}$, it acts as translations on $P_u$. Thus $P_uu=P_{u'}u'$.   
   \end{proof}

  \begin{corollary} \label{Pconstantcorollary}
  Let $S$ be a stratum of $A_0$. Then $\{Au\colon A$ is maximal $(\sigma,\theta_u)$-split$\}$ is independent of $u\in S$.
   \end{corollary}

   \begin{definition}
We say that a maximal $(\sigma,\theta_u)$-split and $\delta$-stable torus $A$ is \emph{standard\/} if $A\cap U\subset A_0$.
\end{definition}

\begin{proposition}\label{prop:standard}
Let $A$ be maximal $(\sigma,\theta_u)$-split where $u\in A_0$. Then there is an $h\in H_0$ such that $h*Au$ is standard.
\end{proposition}

\begin{proof}   
The torus $A$ is maximal $\sigma$-split in $G^{(u)}$, and it follows from \cite[Remark after Lemma 5]{Oshima-Matsuki80} that we may conjugate $A$ by an  element of $((H_0)_u)^0$ such that it becomes $\delta$-stable.  So we may assume that $A$ is $\delta$-stable. Write $A=BC$ where $B$ is a $\delta$-split torus and $C$ is  a compact torus.  By Theorem \ref{thm:generalslice},  the intersection of any orbit $H_0*cu$ with $Cu$ is finite, $c\in C$. Thus the dimension of the intersection of $H_0*Cu$ with $A_0$ is the dimension of $C$. Hence there is a stratum $S$ of $A_0$ of dimension at least  $\dim C$ such that the orbit of an open subset of $Cu$  intersects $S$. Now pick  a  $cu$ in the open set. We can also assume that there is a  $b\in B$ such that $bcu$ is a principal point of $Au$. Thus we can reduce  to the case that $u$ lies in a stratum $S$ of dimension at least $\dim C$ and that there is a principal point $x:=bu$, $b=\exp(Z)\in B$ where $Z\in\SS_u\cap\lier_0$. 
 
 From the proof of \ref{thm:main} we see that the Lie algebra of the unique $(\sigma,\theta_{x})$-split maximal torus through $x$ is $\{Y\in\SS_u\mid [Z,Y]=0\}$, which is, of course, $\liea$. Now if $a\in A_0$ and $au\in S$, we know that $\theta_{au}(Z)=-Z$ and that $\theta_u(Z)=-Z$. Thus $a$ centralizes $Z$ so that $au\in Cu$ since $S$ is connected. But $\dim S\geq\dim C$. Hence a neighborhood of the identity in $C$ lies in $A_0$. Hence $C\subset A_0$ so that $A$ is standard.
\end{proof}

   \begin{proposition}
   Let $S$ be a   stratum of $A_0$ whose closure intersects the stratum $T$. Let $u\in S$ and let $A$ be a maximally   $(\sigma,\theta_u)$-split  torus in $X$. Then $A$ is $(\sigma,\theta_v)$-split for $v\in T$.
   \end{proposition}
   
   \begin{proof}
  By Corollary \ref{Pconstantcorollary} we may assume  that $v\in T$ is in the closure of $S$. Then clearly $A$ is $\theta_v$-split. 
     \end{proof}

   \begin{corollary}\label{cor:minstrata}
   Let $S_1,\dots,S_r$ be the closed strata of $A_0$ and consider a pair $(A,u)$ where $u\in A_0$ and $A$ is a maximal $(\sigma,\theta_u)$-split  torus.  Then for some $v\in\cup_i S_i$, $A$ is maximal $(\sigma,\theta_v)$-split.
   \end{corollary}
   
    If one only wants to find the maximal tori (appropriately split), then one only has to look at the minimal strata. One also only has to look at standard tori.  However, this is not an algorithm. We give an algorithm after Proposition \ref{prop:Aunique} below.

     From Remark \ref{rem:fibers} we have an $H_0$-equivariant projection $\pi\colon  \M\to X_0$, $\exp(\xi)u\mapsto u$. Let $\pi_0$ denote the induced mapping from $\pi\inv(A_0)\to A_0$. From  Theorem \ref{thm:A0} we have an isomorphism $X_0/H_0\simeq A_0/W^*_0$. Let  $\rho\colon A_0\to  A_0/W^*_0$ be the quotient mapping. Then from our isomorphism  and $\pi$   we obtain a surjection $\gamma\colon \M/H_0\to A_0/W^*_0$. We show that  $\pi_0$ and $\gamma$ are product bundles over appropriate strata.

\begin{theorem}\label{thm:fibration}
Let $S$ be a   stratum of  $A_0$ and let $u\in S$. Then, over $S$, $\pi_0$ is a product bundle   $S\times\exp(S_u\cap\lier_0)\to S$ and, over $\rho(S)$, $\gamma$ is a product bundle     $\rho(S)\times \exp(\SS_u\cap\lier_0)/(H_0)_u\to \rho(S)$.
\end{theorem}

\begin{proof}  
From the description of $\M$ (Proposition \ref{prop:M}) and Proposition \ref{prop:Slocal} we see that $\pi_0\inv(S)\simeq  S\times\exp(\SS_u\cap\lier_0)$. Thus $\pi_0$ is a trivial bundle over $S$. 
Now let $\exp(Z_1)u_1$, $\exp(Z_2)u_2\in\pi_0\inv(S)$ where $u_1$, $u_2\in S$ and $Z_1$, $Z_2\in\SS_u\cap\lier_0$. Suppose that $h\in H_0$ and $h*\exp(Z_1)u_1=\exp(Z_2)u_2$. Then $h*u_1=u_2$, so there is an $h'\in N_0^*$ such that $h'*u_1=u_2$. Note that, since $W_0^*$ permutes the strata, $h'*S=S$. Then $(h')\inv h$ fixes $u_1$, hence it fixes all the points of $S$. Let $W^*_S$ denote the elements of $W_0^*$ which send $S$ into itself modulo the elements fixing $S$. Then $W^*_S$ acts freely on $S$ and $\rho(S)\simeq S/W^*_S$. Moreover,  $\gamma\inv(\rho(S))\simeq \rho(S)\times\exp(\SS_u\cap\lier_0)/(H_0)_u$.
\end{proof}

Now consider $u\in S$ and the group $G_1$ which is the fixed group of $\sigma\delta$ acting on $G^{(u)}$. Then $G_1$ is reductive with Cartan involution $\delta$,  and one easily sees that we have the  Cartan decomposition $G_1\simeq (H_0)_u\exp(\SS_u\cap\lier_0)$. Let $\liet$ be a maximal abelian subspace of $\SS_u\cap\lier_0$ and let $W(S,\liet)$ denote the associated Weyl group (the normalizer of $\liet$ in $(H_0)_u$ divided by the corresponding centralizer). Then 
$\exp(\SS_u\cap\lier_0)/(H_0)_u\simeq \exp(\liet)/W(S,\liet)$ (see, for example, \cite[Theorem 11.11]{Helm-Schwarz01}). We call $W(S,\liet)$ the \emph{Weyl group of  $S$ and $\liet$\/}. It is independent of $u\in S$ and the choice of $\liet$, up to isomorphism (see \ref{prop:a}).  

\begin{corollary}\label{cor:fiber}
Let $S$ be a stratum of $A_0$, $u\in S$, and let $\liet$ be a maximal toral subalgebra of $\SS_u\cap\lier_0$. Then the fibers of $\gamma$ above $\rho(S)$ are isomorphic to $\exp(\liet)/W(S,\liet)$.
\end{corollary}
  
We now have the result that $H$-conjugacy is the same as $H_0$-conjugacy for translated tori.

\begin{theorem} \label{thm:equiv} Let $A_i$ be   maximal $(\sigma,\theta_{u_i})$-split tori, $i=1$, $2$. Let $h\in H$ such that $h*A_1u_1=A_2u_2$. Then there is an $h'\in H_0$ such that $h'*A_1u_1=A_2u_2$.
\end{theorem}

\begin{proof} We can assume that the $A_iu_i$ are standard. Since $h*u_1\in\M$, there is an $h_1\in H_0$ such that $h*u_1=h_1*u_1$. Thus $h*u_1\in A_2u_2\cap U$, hence $h*u_1=a_2u_2$ where $a_2\in A_2\cap U\subset A_0$. Now $\conj(h)$ must send the maximal compact subgroup $A_1\cap U$ of $A_1$ to $A_2\cap U$, so that $h*(A_1u_1\cap U)=A_2u_2\cap U$. Write $h=\exp(Y)h_0$ where $h_0\in H_0$ and $Y\in \lieh\cap \lier_0$. Then by uniqueness of the Cartan decomposition one obtains that $h*=h_0*$ on $A_1u_1\cap U$.  Thus we may assume that $h*$ is the identity on $A_1u_1\cap U$.  Then the projections of   $\liea_1$ and $\liea_2$  to $\lieq\cap\lier_0$ are maximally $\sigma$-split subalgebras commuting with $\liea_1\cap\liea_0$. By Proposition \ref{prop:a}   there is an $h'\in ((H_0)_{u_1})^0$  centralizing   $A_1\cap A_0$ such that  $\Ad(h')\liea_1=\liea_2$.
\end{proof}

\begin{remark} \label{rem:S} In the last line of the proof above, $h'\in (H_0)_{u_1}$ already implies that $h'$ fixes all of the stratum $S$ containing $u_1$, hence that it centralizes $A_1\cap A_0$ (which is generated by $Su_1\inv$).
\end{remark}
From Corollary \ref{cor:weylK0'} we obtain

\begin{corollary}\label{cor:equiv}
Let the $A_iu_i$ be standard maximal $(\sigma,\theta_{u_i})$-split tori, $i=1$, $2$. Then $A_1u_1$ and $A_2u_2$ are equivalent if and only if there is a  $w\in W^*_0$ such that $w*(A_1u_1\cap A_0)=A_2u_2\cap A_0$.
\end{corollary}

\begin{corollary}
Let $A$ be maximal $(\sigma,\theta_u)$-split where $u\in A_0$. Then $W^*_{H}(Au)$ is finite
\end{corollary}

 \begin{proof}
We may assume that $Au$ is standard. Let $h\in W^*_{H_0}(Au)=W^*_{H}(Au)$. Modifying $h$ by    an element  of $W^*_0$ we may assume that $h$ is the identity on $Au\cap A_0$. Then the projection $\liet$ of $\liea$ to $\SS_u\cap\lier_0$ is maximal abelian and $h$ induces an element of $W(S,\liet)$ where $S$ is the stratum containing $u$. Thus $W^*_H(Au)$ is finite. 
\end{proof}

\begin{corollary}
Let the $A_iu_i$ be as in Theorem \ref{thm:main}. Then for each $i$ the mapping from the principal points of $A_iu_i$ to $\quot X{H}$ is the quotient by a free action of $W^*_{H_0}(A_iu_i)$.
\end{corollary}

 \begin{proposition} \label{prop:Aunique}
 Let $Au$ be a standard maximal $(\sigma,\theta_u)$-split torus. Let $B$ denote $A\cap A_0$. Let $S$ be a stratum of $A_0$ such that $Bu\cap S$ has the same dimension as $S$. Then, up to equivalence, $Au$ is determined by $S$. 
   \end{proposition}
  
\begin{proof}
We may assume that $u\in S$. Now $\liea\cap\lier_0$ is maximal $\sigma$-split in   $\SS_u\cap\lier_0$.  By Corollary \ref{cor:bconja} and Remark \ref{rem:S}, $\liea\cap\lier_0$ is unique, up to the action of  $(H_0)_u$. Since $A\cap A_0$ is generated by $Su\inv$, $A$ is determined by $S$, up to equivalence.
\end{proof}   

Now we give an algorithm to find a minimal collection of standard tori $A_iu_i$ whose images cover $\quot X{H}$. Let $S_1,\dots,S_r$ be strata of $A_0$ such that the strata of $A_0/W^*_0$ are the $\rho(S_i)$. Let $S$ be one of the $S_i$ and let $u\in S$. Let $\liet$ be a maximal toral subalgebra in $\SS_u\cap\lier_0$.  If $\dim\liet+\dim S=\dim\quot X{H}$, then let $C$ be the connected subtorus of $A_0$ generated by $Su\inv$ and let $B$ denote $\exp(\liet)$. Then $A:=BC$ is a maximal $(\sigma,\theta_u)$-split torus and  we add $Au$ to our collection. If $\dim\liet+\dim S<\dim\quot X{H}$ do not add anything. By Proposition \ref{prop:Aunique} we obtain a set $\{A_iu_i=B_iC_iu_i\}$ whose images cover $\quot X{H}$. Now it is only a matter of removing one of any pair of tori $A_iu_i$ and $A_ju_j$ whenever $C_iu_i$ and $C_ju_j$ are carried one to the other by an element of  $W_0^*$.  Thus we obtain a minimal collection as in Theorem \ref{thm:main}.

\begin{remark}
Let $S_1,\dots,S_p$ be the open strata of $A_0$ and let $A$ be a maximal standard \st-split torus containing $A_0$. Then, up to equivalence,  $A_1u_1=\dots=A_pu_p=A$.
\end{remark}

\begin{example} \label{sl2example:part2} We revisit Example \ref{sl2example:part1}. Let $u=\sm ab{-b}a\in A_0=\SO(2)$. Since $\sm \lambda 00{\lambda\inv}\in H$ sends $u$ to $\sm {\lambda^2a} b {-b}{\lambda^{-2}a}$, the isotropy group of $H$ at $u$ is $\pm I$ for $u\neq\pm v$ and $H$ for $u=\pm v$. Thus the strata of $A_0$ are the two points $\pm v$ and the two connected components of $A_0\setminus\{\pm v\}$. Since $W_0^*$ is trivial, these are also the strata of $A_0/W_0^*$. The maximal \st-split torus $A_0$ comes from  the open strata and the translated tori $\{\sm a b b a\mid a^2-b^2=1, a>0\}(\pm v)$ correspond to the two closed strata.  
\end{example}

 \section{When $K\neq G^\theta$}
Let $K'$ be an open subgroup of $K=G^\theta$ and let $\beta\colon X'=G/K'\to X$ be the corresponding cover.
Let $\M'$ denote $\beta\inv(\M)$. We show that $\M'$ can serve as a Kempf-Ness set for the $H$-action on $X'$.

\begin{lemma}\label{lem:M'} 
Let   $x\in X'$.  
Then the orbit $Hx$ is closed if and only if it intersects $\M'$.
\end{lemma}

\begin{proof}
If $x\in\M'$, then   the $H$-orbits in the closed set $\beta\inv(H\beta(x))$ are open and closed, hence $Hx$ is closed. Conversely, if $Hx$ is closed, then $H\beta(x)$ intersects $\M$, so that $Hx$ intersects $\M'$.
\end{proof}

\begin{lemma}\label{lem:cover}
Let $x\in X'$. Then $\overline{Hx}\cap\M'$ is a single $H_0$-orbit.
\end{lemma}

\begin{proof}
Since $H_0$ acts transitively on the connected components of $H$, we can reduce to the case that $H$ is connected.  For $y\in X$, let $f(y)$ denote the norm squared of $\mu(y)$ using the invariant inner product $B$ of Lemma \ref{lem:innerproduct}, and consider the flow $\psi_t$ on $X$ which is the negative of the gradient of $f$. (We can use the inner product on $X$ induced by the K\"ahler form on $U\c$ \S\ref{sec:moment}.) \ Then $\psi$ is a flow along $H$-orbits and it is $H_0$-equivariant. Moreover, since all our data is real analytic, the flow extends to $+\infty$ to give a deformation retraction of $X$ onto $\M$ (see \cite{SchRutgers} and \cite{HeinznerStoetzel}).  Now the flow lifts to $X'$ to give a deformation retraction $\psi_t'$ of $X'$ onto $\M'$. If $x\in X'$, then $\psi'_\infty(\overline{Hx})$ is connected. We know that the lemma holds for $\beta(x)$ in place of $x$, so that $\overline{Hx}\cap\M'$ is a finite union of $H_0$-orbits which cover   $\overline{H\beta(x)}\cap\M$. But connectedness now implies that there is only one $H_0$-orbit.   
\end{proof}

\begin{corollary}\label{cor:thmsX'}
The analogues of Theorems \ref{thm:orbitclosure} and \ref{thm:homeom} hold for $X'$ and $\M'$ and we have a quotient $\quot {X'}H$. 
\end{corollary} 

Let $X_0'$ denote $G_0/K_0'$ where  $K_0'=K'\cap U$. Then $\beta\colon X_0'\to X_0$ is a cover of the same degree as $X'\to X$. Then one easily shows

\begin{corollary}
 $\M'=\{\exp(Z)v \mid v\in X_0'\text{ and }Z\in\SS_{\beta(v)}\cap\lier_0\}$.
\end{corollary}

\begin{remark}
Let $v$, $v'\in A_0K_0'$, let $Z\in\SS_{\beta(v)}\cap\lier_0$ and let $Z'\in\SS_{\beta(v')}\cap\lier_0$. Then
$\exp(Z)vK'=\exp(Z')v'K'$ if and only if $Z=Z'$ and $v=v'$.
\end{remark}

 Let $S$ be a stratum of $A_0$, and let $S'$ denote $\beta\inv(S)\cap A_0K_0'$. Then $S'=\cup_{j=1}^r S'_j$ where the $S'_j$ are connected and are covering spaces of $S$. Let $u\in S$ and $h\in (H_u)_v$, $v\in S'_j$ for some $j$. Then $h$ preserves $S'_j$ and if $h$ does not act trivially on $S'_j$, then, since $S'_j\to S$ is an $(H_u)_v$-equivariant cover, it follows that $h$ acts nontrivially on $S$, a contradiction. Thus $(H_u)_v$ acts trivially on $S'_j$ and is independent of $v\in S'_j$, so that  we can use  $\cup_S\cup_jS'_j$ as strata for $A_0K_0'$.
  Let $A$ be $(\sigma,\theta_{\beta(v)})$-split and $\delta$-stable where  $v\in A_0K_0'$.   We say that $Av$ (or $A$) is \emph{standard\/} if    $A\cap U\subset A_0$. Then as in Proposition \ref{prop:standard} one shows that for every maximal $(\sigma,\beta(v))$-split  torus $A$, $v\in A_0K_0'$, there is an $h\in H_0$ such that $hAv=hAh\inv hv$ is standard.  

Let $\pi'\colon \M'\to X_0'$ be the canonical map, and let $\pi_0'\colon(\pi')\inv(A_0K_0')\to A_0K_0'$ be the induced mapping.  Let  $\gamma'$ be the canonical mapping of $\M'\to X_0'/H_0\simeq (A_0K_0')/W_0'$ and let $\rho'\colon A_0K_0'\to A_0K_0'/W_0'$ be the quotient mapping. Then we have  the analogue of Theorem \ref{thm:fibration}:  

\begin{theorem}\label{thm:fibration2}
Let $S'$ be a   stratum of  $A_0K_0'$ and let $v\in S'$. Then, over $S'$, $\pi_0'$ is a product bundle   $S'\times\exp(S_{\beta(v)}\cap\lier_0)\to S'$ and, over $\rho'(S')$, $\gamma$ is a product   bundle $\rho'(S')\times \exp(\SS_{\beta(v)}\cap\lier_0)/(H_0)_v\to \rho'(S')$.
\end{theorem}

There is also clearly the analogue of Corollary \ref{cor:fiber}, whose statement we omit. For  $A$ maximal $(\sigma,\beta(v))$-split, let $W_H'(Av)$ denote the group of automorphisms of $Av$ coming from elements of $H$, and define $W'_{H_0}(Av)$ similarly. Then the obvious analogues of Theorem \ref{thm:equiv} and Corollary \ref{cor:equiv} hold and the group $W'_H(Av)=W'_{H_0}(Av)$ is finite.

If one wants to find translated tori $Av$ which minimally cover $\quot {X'}H$, then one can proceed  as before (discussion after Proposition \ref{prop:Aunique}). Here is another method. Let $Au$ be one of the maximal tori occurring in Theorem \ref{thm:main}. We can assume that it is standard. It is covered by tori $Av_j$ where $\beta(v_j)=u$, $v_j\in A_0K'$.   Let $W_{0,u}'$ denote the subgroup of $W_0'$ whose image in $W_0^*$ fixes $(A\cap A_0)u$. Then $W_{0,u}'$ acts on $\beta\inv(u)\cap A_0K'$ and one chooses a maximal subcollection of the $Av_j$ where the $v_j$  are on disjoint $W_{0,u}'$ orbits. Then starting with a minimal collection $A_iu_i$ as in Theorem \ref{thm:main} one arrives at a minimal collection $A_iv_{ij}$
which surjects onto $\M'/H_0\simeq\quot{X'}H$.

\section{Stratification via a Weyl group}
   Now we would like to  compare our stratification  of $A_0$ with that given by a Weyl group. It turns out that $W^*_{H}(A_0)=W^*_0$ is not large enough for our purposes  (see Example \ref{sl2example:part3} below). Thus we have to consider the action of  the   Weyl group $W:=W^*_{H\c}(A_0)$ where  $H\c$ is the Zariski closure of $H$ in $U\c$. Since $e\in\M$, the orbits $H*e$ and $H_\C*e$ are closed and the argument of Remark \ref{rem:w0*finite} shows that  $W$ is finite. 

\begin{proposition} \label{prop:finerstrata} Let $\{S_i\}$ be the set of connected $H$-isotropy type strata of $A_0$  and let $\{T_j\}$ be the  connected strata for the action of $W=W^*_{H\c}(A_0)$.
 Then $\{T_j\}$ is a refinement of $\{S_i\}$, i.e., every $S_i$ is a union of some of the $T_j$.
  \end{proposition}
 
 \begin{proof}
 Let $u\in A_0$. We have to show that the stratum of $A_0$ containing $u$ for the $H$-isotropy stratification contains the corresponding stratum  for the $W$-isotropy stratification near the point $u$. We can consider everything in the transversal at the point $u$, so we replace $G$ by $G^{(u)}$, $U\c$ by $U\c^{(u)}$ and $H\c$ by $H\c \cap U\c^{(u)}$. Thus we may assume that $u=e$ and  that $\sigma=\theta$. Now the complexification $A\c$ of $A$ is a maximal $\sigma$-split torus in $U\c$. The subgroup $W_e$ of $W$ is  generated by elements of $H\c$ which preserve $A_0$ under conjugation.   Near $e$, the stratum of $H$ is $B_1$ where $B_1:=\{b\in A_0\mid  H  \cdot b=b\}$ and the stratum of $W_e$ is $B_2$ where  $B_2:=\{b\in A_0\mid W_e\cdot b=b\}$. 
 We need to show that $B_2^0\subset B_1^0$.
 
 Consider $B_2$ and let $\alpha\colon A_0\to S^1$ be a root of the $A_0$-action on  $\lieg$.  Let $A_0 \subset A$ where $A$ is maximal $\sigma$-split in $G$.
 We show that there is an element of $W_e$ which sends $\alpha$ to $\alpha\inv$, as follows. The root $
\alpha$ is the restriction to $A_0$ of  a root of $A$ and a root of the complexification $A\c$, which we will also call $\alpha$. Let $(A\c)_\alpha$ be the kernel of $\alpha$  and consider the centralizer  of $(A\c)_\alpha$ in $U\c$ divided by $(A\c)_\alpha$. Call this group $C$. Then $C$ has maximal $\sigma$-split torus $A\c/(A\c)_\alpha$, and the roots of this torus acting on $\lie c$ are multiples of $\alpha$. Then there is a $c\in (C\cap H\c)^0$ which  acts as inverse on $A\c/(A\c)_\alpha \simeq\C^*$ \cite{Rich82a}. Thus $c$ gives us an element of $W_e$ which fixes $(A_0)_\alpha$ and acts as inverse on $A_0/(A_0)_\alpha$. It follows that 
  $B_2^0\subset\Ker\alpha$. Since $\alpha$ is an arbitrary root of $A_0$, we find that   $B_2^0$ centralizes $\lieu\c$, hence $B_2^0$ is fixed by $H\c^0$.  If $h\in H$, then $h$  sends $A$ to a maximal $\sigma$-split subtorus $\conj(h) A$ of $G$. Since $A$ is maximally compact, by Corollary \ref{cor:bconja} there is an element of $(H_0)^0$ which  when composed with $h$ stabilizes $A$. Thus, modulo $H^0$ (which fixes $B_2^0$), we may assume that $h\in N_{H}(A)$. Thus $h$ preserves $A_0$ and it acts on $B_2^0$ as an element of $W_e$, i.e., trivially. Thus $B_2^0\subset B_1^0$.  
  \end{proof}
  
  If $K'$ is an open subgroup of $K$, then define $W':=W'_{H\c}(A_0K_0')$ as the subgroup of $H_\C$ stabilizing $A_0K_0'$ divided by the subgroup centralizing $A_0K_0'$. Then $W'$ is finite since $W'\to W$ has finite kernel and $W$ is finite.  
   \begin{corollary}
Let $\{S_i'\}$ be the set of connected $H$-isotropy type strata of $A_0K_0'$  and let $\{T_j'\}$ be the  connected strata for the action of $W'$.
 Then $\{T_j'\}$ is a refinement of $\{S_i'\}$, i.e., every $S_i'$ is a union of some of the $T_j'$.
 \end{corollary}

\begin{proof} 
Let $v\in A_0K_0'$ and set $u=\beta(v)$. Set $C_1=(A_0K_0')^{H_v}$ and $C_2=(A_0K_0')^{W'_v}$. We need to show that $C_2^0\subset C_1^0$. By Proposition \ref{prop:finerstrata} we have that $B_2^0\subset B_1^0$ where $B_2=A_0^{W_u}$ and $B_1=A_0^{H_u}$. 
Note that $[H_u:H_v]$ is finite and that $C_1^0=(B_1^{H_v})^0K_0'$.  We have the analogous result for $C_2^0$.
If $h\in (H_u)^0$, then $h$ fixes   $v$ and $B_2^0$. Clearly   every element of $H_v/(H_u)^0$ is represented by an element normalizing $A_0$. Thus $H_v$ acts trivially on $((B_2K_0')^{W'_v})^0=C_2^0$ so that $C_2^0\subset C_1^0$.
\end{proof}

    \section{Some examples}

 \begin{example} \label{sl2example:part3}. We reexamine  Example \ref{sl2example:part1}, \ref{sl2example:part2} using the Weyl group.   The group $W_0^*$ is trivial, but $W^*_{H\c}(A_0)$ is generated by the element $\sm i 0 0 {-i}$ which acts on $A_0$ by sending $\sm a b {-b} a $ to $\sm {-a} b {-b} {-a}$.  The fixed points of this action are $\{\pm v\}$, hence we obtain again the isotropy type stratification of   $H$. 
\end{example}

 \begin{example}\label{ex:extremes}
    Suppose that $A_0=\{e\}$. By \ref{cor:bconja} there is only one $(H_0)^0$-conjugacy class of maximal \st-split tori in $G$. Thus $\M=H_0*A$ where $A$ is maximal \st-split. The argument of Proposition \ref{prop:w*Y} shows that $ \M/H_0\simeq A/W^*_{H_0}(A)\simeq \quot XH$. Now suppose that $\sigma$ and $\theta$ commute. Then for $h\in N_{H_0}^*(A)$, $\beta(h)\in A$ is an element of order 2, hence it is trivial. Thus $W^*_{H_0}(A)=W_{M_0}(A)$ is an ordinary Weyl group where $M_0=H_0\cap K_0$. This agrees with \cite[11.11]{Helm-Schwarz01} where it is assumed that $\theta$ is a Cartan involution. 
    
    Now suppose that every maximal \st-split torus in $G$ is compact. Then again there is only one $(H_0)^0$- conjugacy class of maximal \st-split tori and   we have $\quot XH\simeq A/W^*_{H_0}(A)$ where $A$ is maximal \st-split. This agrees with the case $G$ compact \S\ref{sec:compact}.
    \end{example}
    
    \begin{example}
Here is a case where $A_0$ is of dimension two and the group $W^*_0$ is   large. Let $U\c=\SL(8,\C)$, $U=\SU(8,\C)$ and $G=\SL(8,\R)$. Let $\delta$ be the associated Cartan involution $g\mapsto {}^t\bar g\inv$, $g\in G$.  Then $G_0=\SO(8,\R)$. The Lie algebra $\lier_0$ consists of the symmetric real matrices. Let $\theta$ be $\delta$ followed by conjugation with $\sm 0 I {-I} 0$ and let $\sigma$ be conjugation with $\sm I 0 0 {-I}$ where   $I$ is $4\times 4$. Then $\sigma$, $\theta$ and $\delta$ commute. The group $H$ is the set of real matrices $\sm A 0 0 B$ such that $\det AB=1$, and being in $H_0$ adds the condition that $A$ and $B$ are orthogonal. The group $K$ is a copy of $\Sp(8,\R)$. The dimension of the quotient of $G$ by $H$ and $K$ is the dimension of the quotient of $\liep\cap\lieq$ by the action of $M=H\cap K$. Now $M$ consists of matrices $\sm g 0 0 {{}^tg\inv}$ where $g\in\GL(4,\R)$ and $\liep\cap\lieq=\{\sm 0 A B 0\mid A$ and $B$ are skew symmetric$\}$. The representation of $\GL(4,\R)$ is isomorphic to that on $\wedge^2\R^4\oplus\wedge^2(\R^4)^*$ which gives us that the quotient has dimension $2$.

Let $J$ denote $\sm 0 1{-1}0$. Then 
$$
A_0=\left\{\left(\begin{matrix} aI & 0 & bJ & 0\\ 0 & a'I & 0 & b'J\\bJ & 0 & aI & 0\\0 & b'J & 0 & a'I\end{matrix}\right)\mid a^2+b^2=1,\ (a')^2+(b')^2=1\right\}
$$
is a maximal \st-split torus (since its dimension is 2) which is compact. Let $a(\eta_1,\eta_2)$ denote the element of $A_0$ as above with $a+ib=e^{i\eta_1}$ and $a'+ib'=e^{i\eta_2}$. Now we compute $W_0^*$.  If $\sm g 0 0 h\in H_0$, then it acts on an element $a\in A_0$ sending it to $\sm g 0 0 ha\sm {h^t} 0 0 {g^t}$. First suppose that $g=\sm \alpha 0 0 I$ and $h=\sm \alpha 0 0 I$ where $\alpha=\sm 0 1 1 0$.  Then $\sm g 0 0 h$ sends $a(\eta_1,\eta_2)$ to $a(-\eta_1,\eta_2)$. If we have $g=\sm J 0 0 I$ and $h=\sm {-J} 0 0 I$, then $a(\eta_1,\eta_2)$ is sent to $a(-\eta_1+\pi,\eta_2)$. Thus $W^*_0$ contains the  translation $a(\eta_1,\eta_2)\mapsto a(\eta_1+\pi,\eta_2)$. Of course we have elements of $W^*_0$ which do similar things to $\eta_2$ while leaving $\eta_1$ fixed. Finally, let $g=h=\sm 0 I I 0$. Then the action of this element interchanges $\eta_1$ and $\eta_2$. Thus we see that $W^*_0$ contains the semidirect product $W(\mathsf B_2)\ltimes(\Z/2\Z)^2$ where $W(\mathsf B_2)$ is the Weyl group of type $\mathsf B_2$ and $(\Z/2\Z)^2$ consists of the pure translations in $W^*_0$. We  claim that $W^*_{H\c}(A_0)$ is no bigger than this. Now the set of pure translations $(\Z/2\Z)^2$   is as large as possible (the pure translations have to be of order $2$ \cite[1.9]{Helm-Schwarz01}).  Thus the remaining elements of $W^*_{H\c}(A_0)$ fix $e\in A_0$, i.e., they are fixed by $\theta$. Moreover, the Weyl group has representatives which are in $U$, so that we are reduced to calculating $W_{M\cap U}(A_0)$. Now $M\cap U=\{\sm g 0 0 g\mid g\in\U(4,\C),\ \det g=\pm 1\}$. A matrix calculation now shows that one cannot obtain elements of the Weyl group of $A_0$ that we have not already seen. Thus $W^*_{H\c}(A_0)=W^*_0\simeq W(\mathsf B_2)\ltimes(\Z/2\Z)^2$.

Now a fundamental domain $D$ for the action of $W^*_0$ on $A_0$ consists of the elements $a(\eta_1,\eta_2)$ where $0\leq\eta_1\leq \eta_2\leq\pi/2$. Since $W^*_{H\c}(A_0)=W^*_0$, the strata for the action of $W^*_0$ and $H$ are the same. We now list the various strata $S$ occurring in $D$.

\smallskip
\noindent Case 1: $S=\{a(\eta_1,\eta_2)\mid0<\eta_1<\eta_2<\pi/2\}$ is the only two-dimensional stratum. The preimage of $S$ in the quotient $\quot X{H}$ is just a copy of $S$.

\smallskip
\noindent Case 2: $S=\{a(\eta_1,\eta_2)\mid 0<\eta_1=\eta_2<\pi/2\}$. Here the preimages of points of $S$ are one-dimensional. Let $a\in S$. Then  $(H_0)_a=\{\sm g 0 0 g\mid g$ is orthogonal and $gJ_2g\inv=J_2\}$ where $J_2=\sm J 0 0 J$ gives a complex structure on $\R^4$. Thus $(H_0)_a\simeq \U(2,\C)$. This group acts on $\SS_a\cap\lier_0$ which consists of matrices $\sm 0 C {-C} 0$ where $C$ is skew symmetric and $J_2C=-CJ_2$. This set of matrices is a complex vector space of dimension $2$, and we have the standard action of $\U(2,\C)$ on $\C^2$. The Lie algebra $\liet$ of a maximal torus  in $\exp(\SS_a\cap\lier_0)$ is generated by $\sm 0 C {-C} 0$ where $C=\sm 0\alpha {-\alpha} 0$ and $\alpha=\sm 0 I I 0$. Note that $\liet$ is stable under the action of $(W^*_0)_a\simeq\Z/2\Z$ (generated by the element reversing the order of $\eta_1$ and $\eta_2$), and the action is multiplication by $-1$.  We have $W(S,\liet)\simeq\Z/2\Z$,  and   $\exp(\SS_u\cap\lier_0)/(H_0)_a\simeq \R^{>0}/(\Z/2\Z)$ where the Weyl group action sends $r\in \R^{>0}$ to $1/r$. Thus the preimage of $S$ is isomorphic to $S\times\{r\in\R\mid r\geq 1\}$.

\smallskip
\noindent Case 3: $S=\{a(0,\eta_2)\mid 0 < \eta_2<\pi/2\}$ or $S=\{\eta_1,\pi/2)\mid 0<\eta_1<\pi/2\}$. As above,   a maximal  abelian subspace $\liet\subset\SS_a\cap\lier_0$ has dimension 1 and $W(S,\liet)\simeq\Z/2\Z$.  The fiber above any point $a\in S$ is isomorphic to $\{r\in\R\mid 1\leq r\}$.  

 \smallskip
\noindent Case 4: $S=\{e=a(0,0)\}$.    From our calculations above we have that $(H_0)_e=H_0\cap K_0=\{\sm g 0 0 g\mid g$ is orthogonal$\}$ and that $\SS_e\cap\lier_0=\liep\cap\lieq\cap\lier_0=\{\sm 0 C {-C} 0\mid C$ is skew symmetric$\}$. Thus we have the adjoint representation of $\Orth(4,\R)$. We choose as maximal toral subalgebra $\lie t$ the set of matrices
$$ 
\left(\begin{matrix} 0 & 0 & t_1J & 0\\ 0 & 0 & 0& t_2J\\-t_1J & 0 & 0 & 0\\ 0 & -t_2J & 0 & 0\end{matrix}\right).
$$
Then $(W_0^*)_e\simeq W(\mathsf B_2)\simeq W(S,\liet)$ acts on $\lie t$ in the standard way. Exponentiating we get pairs of positive real numbers $r_1$ and $r_2$ whose normal form under $W(\mathsf B_2)$ consists of $\{(r_1,r_2)\mid 1\leq r_1\leq r_2\}$. This is the preimage of $S$.

\smallskip
\noindent Case 5: $S=\{a:=a(\pi/2,\pi/2)\}$. Here $(H_0)_a$ consists of the matrices $\sm g 0 0 h$ such that $g$ and $h$ are orthogonal and commute with $J_2$. Thus our group is isomorphic to $\U(2,\C)\times \U(2,\C)$. The vector space $\SS_a\cap\lier_0$ consists of matrices  $\sm 0 C {C^t} 0$ where $C$ anticommutes with $J_2$. This is a complex vector space of dimension 4, and $\sm g 0 0 h$ acts sending $C$ to $gCh\inv$. Thus we have the obvious action of $\U(2,\C)\times \U(2,\C)$ on $\C^2\otimes_\C \C^2$. As maximal toral subalgebra we can choose $\lie t=\{\sm 0{t_1\alpha} {t_2\alpha} 0\}$ where $\alpha=\sm 0 11 0$ and $t_1$, $t_2\in\R$. Here  $(W_0^*)_a$ once again acts on $\liet$, but not faithfully. The image only contains the element interchanging $t_1\alpha$ and $t_2\alpha$. One easily calculates that $W(S,\liet)\simeq W(\mathsf B_2)$. Hence $\exp(\SS_a\cap\lier_0)/(H_0)_a\simeq\{(r_1,r_2)\in\R^2\mid 1\leq r_1\leq r_2\}$. This is the same fiber that we saw in Case 4.

\smallskip
\noindent Case 6: $S=\{a:=a(0,\pi/2)\}$. Here the story is slightly different. Again $(W^*_0)_a$ acts on an appropriately chosen  maximal abelian subspace $\liet\subset  \SS_a\cap\lier_0$ where $\liet\simeq\R^2$.  The image of $(W^*_0)_a \simeq(\Z/2\Z)^2$ in  $W(S,\liet)$ acts as a sign change  on one of the coordinates while $W(S,\liet)\simeq(\Z/2\Z)^2$ acts by sign changes of both coordinates.  Thus the fiber above $a$ is  naturally isomorphic to $\{(r,s)\in\R^2\mid r,\ s\geq 1\}$.
\end{example}


\newcommand{\noopsort}[1]{} \newcommand{\printfirst}[2]{#1}
  \newcommand{\singleletter}[1]{#1} \newcommand{\switchargs}[2]{#2#1}
  \def\cprime{$'$}
\providecommand{\bysame}{\leavevmode\hbox to3em{\hrulefill}\thinspace}
\providecommand{\MR}{\relax\ifhmode\unskip\space\fi MR }
\providecommand{\MRhref}[2]{%
  \href{http://www.ams.org/mathscinet-getitem?mr=#1}{#2}
}
\providecommand{\href}[2]{#2}

\end{document}